\documentclass[11pt]{amsart}  
\usepackage{enumitem}
\usepackage{shuffle}
\usepackage[latin1]{inputenc}
\usepackage{amsmath}
\usepackage{amsfonts}
\usepackage{amssymb}
\usepackage{stmaryrd}
\usepackage{latexsym}
\usepackage{graphicx}
\usepackage{subfigure}
\usepackage{color}
\usepackage{mathtools}
\usepackage{verbatim}
\usepackage[all]{xy}
\usepackage{graphics}
\usepackage{pdfsync}
\usepackage[backref=page]{hyperref}
\usepackage{ytableau}
\usepackage{xcolor}
\usepackage{tikz}
\usetikzlibrary{arrows,petri,topaths}
\usepackage{tikz-qtree}
\usetikzlibrary{positioning,arrows,calc,intersections,shapes,decorations}
\usepackage[normalem]{ulem}

\theoremstyle{plain}
\newtheorem{theorem}{Theorem}[section]
\newtheorem{proposition}[theorem]{Proposition}

\newtheorem{lemma}[theorem]{Lemma}
\newtheorem{corollary}[theorem]{Corollary}

\theoremstyle{definition}
\newtheorem{definition}[theorem]{Definition}
\newtheorem{example}[theorem]{Example}

\theoremstyle{remark}
\newtheorem{remark}[theorem]{Remark}

\numberwithin{equation}{section}
\newcommand{\mrm}[1]{\mathrm{#1}}

\newcommand{\bN}{\mathbb{N}}

\newcommand{\bP}{\mathbb{Z}_+}
\newcommand{\bQ}{\mathbb{Q}}

\newcommand{\bZ}{\mathbb{Z}}

\newcommand{\tcr}[1]{\textcolor{red}{#1}}
\newcommand{\tcb}[1]{\textcolor{blue}{#1}}

\newcommand{\des}{\mathrm{Des}} 

\newcommand{\suchthat}{\;|\;}

\newcommand{\alpx}{\mathbf{x}}
\newcommand{\alpxm}{\mathbf{x}_{-}}
\newcommand{\alpxp}{\mathbf{x}_{+}}

\newcommand{\backquasi}{\overleftarrow{QR}}

\newcommand{\rev}[1]{\mathrm{rev}(#1)} 
\newcommand{\mc}[1]{\mathcal{#1}} 
\newcommand{\qsym}{\mathrm{QSym}} 

\newcommand{\supp}{\mathrm{Supp}}
\newcommand{\bfc}{{\bf c}}
\newcommand{\stan}{\mathrm{stan}} 

\newcommand{\backslide}{\overleftarrow{\mathfrak{F}}} 
\newcommand{\slide}{\mathfrak{F}} 
\newcommand{\backforest}{\overleftarrow{\mathfrak{P}}} 
\newcommand{\forest}{\mathfrak{P}} 

\newcommand{\Zaug}{\overline{\bZ}}
\newcommand{\Paug}{\overline{\bP}}
\newcommand{\lf}[2]{#1^{[#2]}}         

\newcommand{\injwords}[1]{\mathrm{Inj}(#1)}

\newcommand{\flatten}[1]{\mathrm{flat}(#1)}
\newcommand{\reishi}{\mathrm{RS}}

\newcommand{\pa}{m}

\title{$P$-partitions with flags and back stable quasisymmetric functions}

\author{Philippe Nadeau}
\address{Univ Lyon, Universit\'e Claude Bernard Lyon 1, CNRS UMR
5208, Institut Camille Jordan, 43 blvd. du 11 novembre 1918, F-69622 Villeurbanne cedex, France}
\email{\href{mailto:nadeau@math.univ-lyon1.fr}{nadeau@math.univ-lyon1.fr}}

\author{Vasu Tewari}
\address{Department of Mathematics, University of Hawaii at Manoa, Honolulu, HI 96822, USA}
\email{\href{mailto:vvtewari@math.hawaii.edu}{vvtewari@math.hawaii.edu}}
\thanks{P.~N is partially supported by the project ANR19-CE48-011-01 (COMBIN\'E). V.~T. acknowledges the support from Simons Collaboration Grant \#855592.}

\begin{document}
\begin{abstract}
Stanley's theory of $(P,\omega)$-partitions is a standard tool in combinatorics. 
It can be extended to allow for the presence of a restriction, that is a given maximal value for partitions at each vertex of the poset, as was shown by Assaf and Bergeron. 
Here we present a variation on their approach, which applies more generally. 
The enumerative side of the theory is more naturally expressed in terms of back stable quasisymmetric functions. We study the space of such functions, following the work of Lam, Lee and Shimozono on back stable symmetric functions.
As applications we describe a new basis for the ring of polynomials that we call forest polynomials. 
Additionally we give a signed multiplicity-free expansion for any monomial expressed in the basis of slide polynomials.
\end{abstract}

\maketitle

\section*{Introduction}

The study of $(P,\omega)$-partitions was initiated by Stanley in his thesis \cite{Stanley72} wherein he established his \emph{fundamental theorem} laying the groundwork.
Since then $(P,\omega)$-partitions, and several close cousins thereof, have proven to be very useful in combinatorics. We refer the reader to a recent survey by Gessel \cite{GesSurvey} where the history is nicely recounted; see also \cite[Section 5]{GriRei}.
Several aspects of $(P,\omega)$-partitions manifest themself in the combinatorics of certain generating series $K_{(P,\omega)}$ which turn out to be quasisymmetric functions.
A class that then emerges naturally is when $P$ is a chain, in which case one obtains Gessel's fundamental quasisymmetric functions. An immediate, and pleasant, consequence of Stanley's fundamental theorem is that $K_{(P,\omega)}$ expands \emph{positively} in terms of fundamental quasisymmetric functions, i.e. the expansion involves nonnegative integer coefficients.

A polynomial analogue of the latter was introduced by Assaf and Searles \cite{AS} under the name slide polynomials.
These were then realized by Assaf and Bergeron \cite{AssBer20} in a manner akin to  how fundamental quasisymmetric functions are defined via $(P,\omega)$-partitions.
To this end, the authors in \emph{loc. cit.} introduced the notion of $(P,\omega,\rho)$-partitions, which are $(P,\omega)$-partitions in the traditional sense now subject to upper bound constraints imposed by the \emph{restriction} $\rho$.
The generating series $K_{(P,\omega)}$ from before turn into polynomials $K_{(P,\omega,\rho)}$.
The natural analogue of Stanley's fundamental theorem holds in this restricted setting as well, and Assaf--Bergeron establish that $K_{(P,\omega,\rho)}$ expands positively in terms of slide polynomials when $\rho$ is $\mrm{AB}$-flag; see Section~\ref{sub:restriction} for the precise definition.

The main purpose of this note is to introduce $(P,\Phi)$-partitions, which are a mild generalization of those $(P,\omega,\rho)$-partitions that satisfy the $\mrm{AB}$-flag condition. They form a class that is also easier to deal with, leading to simplified proofs. In particular, we find ourselves in the following desirable setting once more\textemdash{} an analogue of Stanley's fundamental theorem holds and this allows us to describe the set of $(P,\Phi)$-partitions as a disjoint union of $(L,\Phi)$-partition sets as $L$ varies over the set of linear extensions of $P$. This then allows us to express $K_{(P,\Phi)}$ as a sum of the $K_{(L,\Phi)}$, each of which either equals a slide polynomial or $0$. With an eye toward future work, we introduce the basis of forest polynomials.

It turns out that the \emph{back stable} setting is a particularly convenient place to view these results. By allowing our $(P,\Phi)$-partitions to take values in the negative integers, the polynomials $K_{(P,\Phi)}$ become generating series $\overleftarrow{K}_{(P,\Phi)}$ in the alphabet $x_i$ where $i<N$ for some $N\in \bZ$. The $\overleftarrow{K}_{(P,\Phi)}$ naturally live inside the space of \emph{back stable quasisymmetric functions} $\overleftarrow{QR}$, which we define and study in Section~\ref{sub:backstable}. 
In particular we establish that the back stable limits of slide polynomials form a basis for $\overleftarrow{QR}$.

 Finally we give an explicit description expressing a monomial in the basis of slide polynomials. It turns out that the coefficients that arise belong to $\{0,\pm 1\}$. In other words, this expansion is \emph{signed multiplicity-free}.

\section{Combinatorial preliminaries}
\label{sec:preliminaries}

Given a sequence $\bfc=(c_i)_{i\in\bZ}$ of nonnegative integers, we define its \emph{support} $\supp(\bfc)$ to be the set of indices $i$ such that $c_i>0$. 
We call  $\bfc$ an \emph{$\bN$-vector} if  $\supp(\bfc)$ is finite. If $\supp(\bfc)\subseteq [n]\coloneqq \{1,\dots,n\}$ for some positive integer $n$, we occasionally write $\bfc=(c_1,\dots,c_n)$ and refer to $\bfc$ as a \emph{weak composition}. 
For any $\bN$-vector $\bfc$, we let $|\bfc|\coloneqq \sum_{i}c_i$ denote its \emph{weight}.
The finite sequence of positive integers obtained by omitting all $0$s from $\bfc$ is called a \emph{strong composition}.
Henceforth, by composition, we shall always mean strong composition. If $\alpha=(\alpha_1,\dots,\alpha_{\ell(\alpha)})$ is a composition of weight $r\geq 0$, we denote this by $\alpha\vDash r$. Here $\ell(\alpha)$ is the \emph{length} of $\alpha$.
The unique composition of weight and length both $0$ will be denoted by $\varnothing$.
We will need two operations on compositions\textendash{} \emph{concatenation} and \emph{near-concatenation}.
Both take compositions $\alpha=(\alpha_1,\dots,\alpha_{p})$ and $\beta=(\beta_1,\dots,\beta_q)$ as input; the former produces $\alpha \cdot \beta=(\alpha_1,\dots,\alpha_p,\beta_1,\dots,\beta_q)$ whereas the latter produces $(\alpha_1,\dots,\alpha_p+\beta_1,\beta_2,\dots,\beta_q)$.

When writing our $\bN$-vectors $\bfc=(c_i)_{i\in \bZ}$ we use a vertical bar to distinguish the `left half' $(\dots,c_{-1},c_0)$ and the `right half' $(c_1,c_2,\dots)$. As an example, note that $\bfc=(\ldots,0,2|0,1,2,0,2,0,\ldots)$ has $c_{0}=2$ separated from $c_1=0$ by a vertical bar.
We will be particularly interested in $\bN$-vectors supported on $\bZ_+$, in which case we will omit the vertical bar and the zeros left of it.

 We denote the set of words in an alphabet $\mc{A}$ by $\mc{A}^*$. We denote the \emph{length} of any word $w\in \mc{A}^*$ by $\ell(w)$. We let $\epsilon$ denote the \emph{empty word}, i.e. the unique word of length $0$.
A subset of $\mc{A}^*$ which we care about is  $\injwords{\mc{A}}$, the set of \emph{injective words}, i.e. words with  all distinct letters.

We now proceed to define the ordered alphabet that is relevant for our purposes. It is obtained by ``augmenting'' $\bZ$.
Denote the set of positive integers by $\bZ_+$.

\begin{definition}
Let $\Zaug$ be the ordered alphabet with letters $\lf{i}{j}$ where $i\in \bZ$ and $j\in \bP$. We have a linear order on $\bZ\sqcup \Zaug$ given by $i<\lf{i}{1}<\lf{i}{2}<\lf{i}{3}<\cdots<i+1$ for all $i$.
\end{definition}

The order is the lexicographic order on $\bZ\times \bP$, but the notation will serve to highlight the prevalent role played by the first factor. We define the \emph{value} of $\lf{i}{j}$ by $\mathrm{val}(\lf{i}{j})=i$.

There is a natural way to go from $\bZ^*$ to $\injwords{\Zaug}$: given $w$ in $\bZ^*$, one associates a word $W=\stan(w)$  by labeling the occurrences of the same letter $i$ in $w$ from left to right by $\lf{i}{1},\lf{i}{2},\dots$. For instance $w=1221625$ gets mapped to $W=\lf{1}{1}\lf{2}{1}\lf{2}{2}\lf{1}{2}\lf{6}{1}\lf{2}{3}\lf{5}{1}$.
This process is clearly injective, and will be used as our natural embedding
\begin{equation}
\label{eq:Zembedding}
\stan:\bZ^*\hookrightarrow \injwords{\Zaug}.
\end{equation}

\section{Revisiting Stanley and Assaf--Bergeron}
\label{sec:whats_known}

We first recall the setting and main results of the celebrated theory of $(P,\omega)$-partitions due mainly to Stanley. We then explain how this theory extends in the presence of a {\em restriction} by recalling pertinent results of Assaf and Bergeron.

In this section $(P,\leq_P)$ is a finite poset. 
We will interchangeably use $P$ to denote both the poset as well as the set underlying the poset.
 We denote the cover relation by $\prec_P$. 
 
 \begin{definition} \label{def:Ppartition} A \emph{$P$-partition} is a function $f:P\to \bP$ such that $f(u)\geq f(v)$ whenever $u\prec_P v$. 
\end{definition}

\subsection{Stanley's theory of $(P,\omega)$-partitions.}
\label{sub:stanley}

 The starting point is to fix in addition a bijective labeling $\omega:P\to \{1,\ldots,\# P\}$. The pair $(P,\omega)$ then forms a labeled poset. A \emph{$(P,\omega)$-partition} is  a $P$-partition $f$ such that $f(u)>f(v)$ whenever $u\prec_P v$ and $\omega(u)>\omega(v)$. 
 Let $\mrm{Part}(P,\omega)$ denote the set of $(P,\omega)$-partitions.
 
  A \emph{linear extension} $L$ of $P$ is a linear ordering of $P$ extending $\leq_P$. Thus a linear extension is a totally ordered set with $P$ as its underlying set. Let $\mathrm{Lin}(P)$ be the set of  linear extensions of $P$. For example, the $(P_0,\omega_0)$-partitions for the example illustrated in Figure~\ref{fig:small_poset} are the functions $f:P_0\to\bP$ that satisfy $f(a)\geq f(b)$ and $f(c)>f(b)$.
  Throughout this note, we depict strict inequalities using dashed edges in Hasse diagrams.

\begin{figure}[!ht]
\includegraphics[]{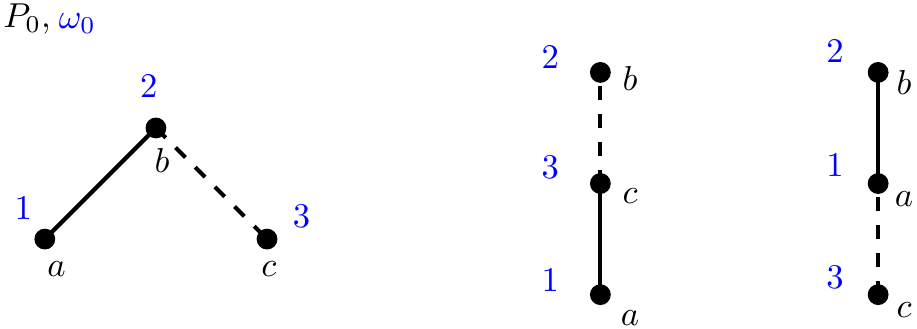}
\caption{
A labeled poset and its two linear extensions.
}
\label{fig:small_poset}
\end{figure}

\begin{theorem}[{Stanley's fundamental theorem \cite[Theorem 6.2]{Stanley72}}]
\label{theo:stanley}
We have
\[
\mrm{Part}(P,\omega)=\displaystyle\bigsqcup_{L\in \mrm{Lin}(P)} \mrm{Part}(L,\omega).
\]
\end{theorem}

For $(P_0,\omega_0)$ in Figure~\ref{fig:small_poset}, the two linear extensions are shown together with their induced labeling on the right. Theorem~\ref{theo:stanley} says that  $(P_0,\omega_0)$-partitions are the functions $f$ satisfying either $f(a)\geq f(c)>f(b)$ or $f(c)>f(a)\geq f(b)$, as can be directly checked.

Now for any  $(P,\omega)$ one can consider the generating function
\begin{equation}
\label{eq:KPomega}
K_{(P,\omega)}=\sum_f \prod_{u\in P} x_{f(u)}
\end{equation}
where the sum is over the set of all $(P,\omega)$-partitions. 
Then Stanley's theorem has the following expansion as an immediate corollary:
\begin{align}
\label{eq:stanley_fundamental}
K_{(P,\omega)}=\sum_{L\in\mathrm{Lin}(P)}K_{(L,\omega)}.
\end{align}
What makes this particularly useful is that the series $K_{(L,\omega)}$ are precisely the fundamental quasisymmetric functions introduced by Gessel~\cite{Ges84}, as we now detail. 

Recall that a series $f$ in the variables $x_i$ for $i\in I$, with $I$ an interval in  $\bZ$, is {\em quasisymmetric} if for any composition $\alpha=(\alpha_1,\ldots,\alpha_k)$ and any subsets $\{i_1<\cdots<i_k\}$ and $\{j_1<\cdots<j_k\}$ of $I$, the coefficients  of $x_{i_1}^{\alpha_1}\cdots x_{i_k}^{\alpha_k}$ and $x_{j_1}^{\alpha_1}\cdots x_{j_k}^{\alpha_k}$ in $f$ are the same. 

Let $\alpxp$ denote the set of \emph{positively-indexed variables} $\{x_i\suchthat i\in \bZ_{>0}\}$. Given a subset $S$ of $\{1,\ldots,r-1\}$, the \emph{fundamental quasisymmetric function} $F_{r,S}\in\bQ[\![\alpxp]\!]$ is defined by
\begin{align}
\label{eq:fundamental_set}
F_{r,S}=\sum_{\substack{1\leq i_1\leq \cdots \leq i_r\\i_j<i_{j+1} \text{ if }j\in S }}x_{i_1}\cdots x_{i_r}.
\end{align}

For instance 
  $F_{3,\{1\}}=\sum_{1\leq i_1<i_2 \leq i_3}x_{i_1}x_{i_2}x_{i_3}$ and $F_{3,\{2\}}=\sum_{1\leq i_1\leq i_2 < i_3}x_{i_1}x_{i_2}x_{i_3}$. The $F_{r,S}$ for $r\geq 0$ and $S$ a subset of $\{1,\ldots,r-1\}$ form a basis of the space of quasisymmetric functions in $\alpxp$. Let $\alpha=(\alpha_1,\ldots,\alpha_\ell)\vDash r$ be the composition corresponding to $S\subseteq \{1,\ldots,r-1\}$ under the folklore correspondence given by $S=\{\alpha_1,\alpha_1+\alpha_2,\ldots,\alpha_1+\alpha_2+\cdots+\alpha_{\ell-1}\}$. \emph{We will then use freely the notation $F_\alpha$ to denote $F_{r,S}$}. For instance $F_{(1,2)}=F_{3,\{1\}}$ and $F_{(2,1)}=F_{3,\{2\}}$.
  
Now if  $L$ is a chain $v_1\prec_L \cdots \prec_L v_r$ with a labeling $\omega$, define $\des(\omega)=\{ i\in [r-1]\suchthat \omega(v_i)>\omega(v_{i+1})\}$. If we denote the composition of $r$ corresponding to $\des(\omega)$ by $\alpha_\omega$, then it is easily verified that
\[
K_{(L,\omega)}=F_{\rev{\alpha_\omega}}.
\]
Here $\mathrm{rev}$ reverses its input. This shows that the expansion~\eqref{eq:stanley_fundamental} expresses any $K_{(P,\omega)}$ positively in the basis of fundamental quasisymmetric functions. 
For the poset in Figure~\ref{fig:small_poset}, we get the expansion $K_{(P_0,\omega_0)}=F_{(2,1)}+F_{(1,2)}$.

\subsection{Restricted partitions}
\label{sub:restriction}

We now want to constrain $P$-partitions to be dominated by certain fixed values at each element of~$P$. Let $(P,\omega)$ be a labeled poset. 

\begin{definition}
\label{def:restriction}
Fix $\rho:P\to\bP$  a map, called a \emph{restriction}. We define \emph{$(P,\omega,\rho)$-partitions} as those $(P,\omega)$-partitions $f$  that satisfy  $f(u)\leq \rho(u)$ for all $u\in P$. 
\end{definition}

Let $\mathrm{Part}(P,\omega,\rho)$ be the set of all $(P,\omega,\rho)$-partitions. Assaf and Bergeron~\cite{AssBer20} have already considered this extension. We recall some of their results (see Remark~\ref{rem:comparison_AB} for easy comparison). 
First they note that  Theorem~\ref{theo:stanley} adapts immediately:

\begin{theorem}[{\cite[Theorem 3.14]{AssBer20}}]
\label{theo:stanley_restricted}
For any restriction $\rho$, we have 
\[
\mrm{Part}(P,\omega,\rho)=\displaystyle\bigsqcup_{L\in \mrm{Lin}(P)} \mrm{Part}(L,\omega,\rho).
\]
\end{theorem}

So, for instance, if we impose the restriction $\rho(a)=3$, $\rho(b)=\rho(c)=2$ for the poset $P_0$ in Figure~\ref{fig:small_poset}, then $\mrm{Part}(P_0,\omega,\rho)$ comprises exactly two functions: $f(a)=f(c)=1, f(b)=2$ and $f(c)=1, f(a)=f(b)=2$.
The former comes from the linear extension in the middle and the latter from that on the right.

Definition~\ref{def:restriction} leads to a restricted version of ~\eqref{eq:KPomega} allowing us to define generating functions $K_{(P,\omega,\rho)}$, which are now polynomials, in the obvious manner. 
Theorem~\ref{theo:stanley_restricted} then naturally gives a restricted version of~\eqref{eq:stanley_fundamental} for $K_{(P,\omega,\rho)}$.

Now the issue here is that, when $L$ is a linear order, the family of functions $K_{(L,\omega,\rho)}$ for varying $\omega$ and $\rho$ is too large, and is in particular not free. 
There is however a very natural subfamily that forms a basis of the space of all polynomials, namely the \emph{slide polynomials} \cite{AS} to which we will come back later. A remaining issue is that, for general restrictions $\rho$, the expansion of $K_{(P,\omega,\rho)}$ is not necessarily positive in the slide basis, cf. ~\cite[Example 3.12]{AssBer20}.

Assaf and Bergeron are thus led to consider a restricted set of restrictions $\rho$. 
These are defined as follows:  Say that the restriction $\rho$ is an \emph{AB-flag}  for $(P,\omega)$ if the following conditions hold (cf. \cite[Definition 4.1]{AssBer20}). 
\begin{enumerate}[label=(AB\arabic*)]
\item \label{it:AB1} If $u\prec_P v$, then $\rho(u)\geq \rho(v)$;
\item \label{it:AB2} If $u\prec_P v$ and $\rho(u)>\rho(v)$, then $\omega(u)>\omega(v)$.
\end{enumerate}

Their main result is then the following:
\begin{theorem}
\label{thm:AB_main}
 If $\rho$ is an AB-flag for $(P,\omega)$, the polynomial $K_{(P,\omega,\rho)}$ expands in the slide basis with nonnegative integral coefficients.
 \end{theorem}
 We postpone the definition of slide polynomials to the next section, preferring to define them via a seemingly different perspective than their original definition \cite[Definition 3.6]{AS}. \smallskip

 In order to prove Theorem~\ref{thm:AB_main}, Assaf and Bergeron modify $\rho$ to have it satisfy in addition a certain ``well-labeled'' restriction; see \cite[Section 4]{AssBer20}.
This modification is needed because the AB-flag property alone does not transfer to linear extensions as can be easily seen. 
As Assaf and Bergeron demonstrate in \cite[Proposition 4.3]{AssBer20}, the set of  flagged $(P,\rho)$-partitions remains unchanged in spite of these modifications.
To assist the reader following their proof, we note some key aspects. They employ the fact that if one takes the Hasse diagram of any $(P,\omega)$ and removes all strict edges, then the resulting connected components inherit an ordering from $P$. Note that they start with an AB-flag but at this stage this plays no role since a constant AB-flag puts no restriction on $\omega$. 
As the example in Figure~\ref{fig:counter_ex} shows, we can in fact get two connected components that are not ordered.
That being noted, the removal of strict edges does indeed give a preorder on the connected components. 
Now if one considers the induced order, then the AB-flag property forces all elements of these supercomponents to have the same $\rho$-values, and the remainder of their proof goes through.
\begin{figure}[!h]
\includegraphics[scale=0.6]{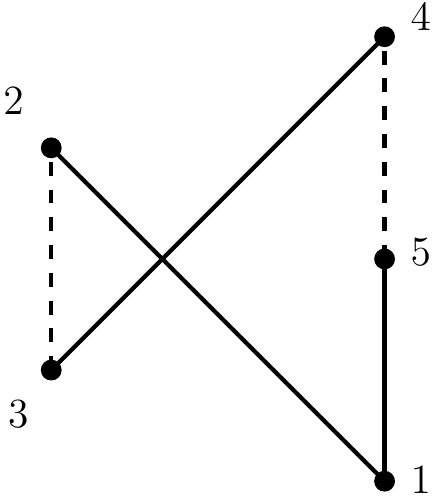}
\caption{}
\label{fig:counter_ex}
\end{figure}

In the next section, we will find a simpler condition on restrictions, which  at the same time is more general than AB-flags and for which (positivity of) the  slide expansion will follow immediately.

\begin{remark}
\label{rem:comparison_AB}
Our setup differs slightly from \cite{AssBer20}.
 First, our $P$-partitions are order-reversing instead of order-preserving. To go between the two formulations, one simply has to go from a poset $(P,\leq_P)$ to its opposite $(P,\geq_P)$.  
Also, the  function $\omega$ is implicit in their work, since they consider the set underlying  $P$ to be $\{1,\ldots,\# P\}$ already. This allows for compact notation, but is not well adapted to having different labelings on a given abstract poset. Furthermore, the forthcoming notion of $(P,\Phi)$-partitions would be cumbersome to use in this setting.
\end{remark}

\section{$(P,\Phi)$-partitions}
\label{sec:whats_new}

We now define our notion of $(P,\omega)$-partitions with restrictions.
 Recall the ordered alphabet $\Paug$ from Section~\ref{sec:preliminaries}, given by letters $\lf{i}{j}$ with $i,j\in\bP$. 

\begin{definition}[Labeled flag $\Phi$]
\label{def:P_phi}
Let $P$ be a poset. A \emph{labeled flag} on $P$  is an \emph{injective} function $\Phi:P\to \Paug$.
\end{definition}

\begin{definition}[$(P,\Phi)$-partitions]
\label{def:P_phi_partition}
Let $(P,\Phi)$ be a poset with a labeled flag. A \emph{$(P,\Phi)$-partition} is a function $f:P\to \bZ_+$ such that for any $u,v\in P$:
\begin{itemize}
\item $f(u)\geq f(v)$ if $u\prec_P v$;
\item $f(u)>f(v)$ if $u\prec_P v$ and $\Phi(u)>\Phi(v)$;
\item $f(u)\leq \mathrm{val}(\Phi(u))$.
\end{itemize} 
\end{definition}
We denote the set of all $(P,\Phi)$-partitions by $\mathrm{Part}(P,\Phi)$. An example is given in Figure~\ref{fig:counter_example}.
On the left is the Hasse diagram of the \emph{diamond} poset with the labeled flag in blue. On the right are the five $(P,\Phi)$-partitions. The $e$-labels on the covers may be momentarily ignored.
\begin{figure}[!h]
\includegraphics[scale=0.65]{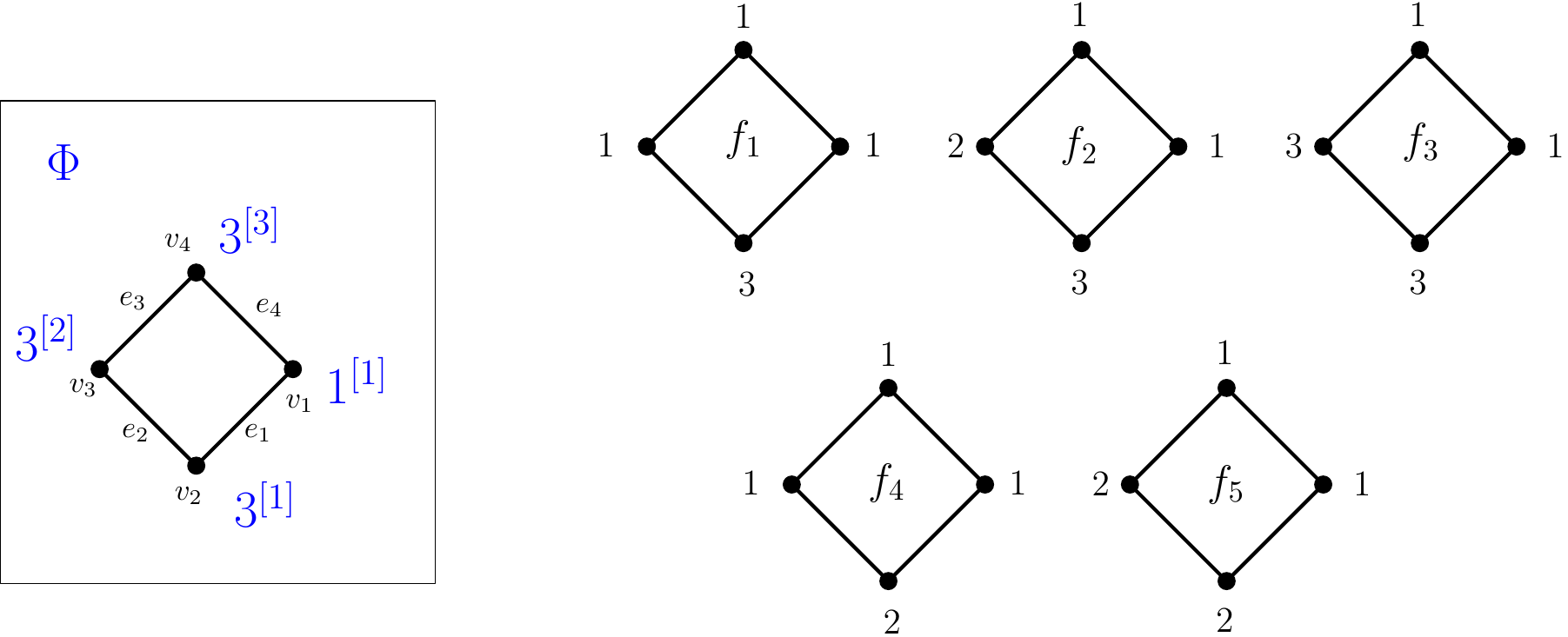}
\caption{A poset with labeled flag $(P,\Phi)$ and its five $(P,\Phi)$-partitions}
\label{fig:counter_example}
\end{figure}

Let us note immediately that the notion of $(P,\Phi)$-partitions is simply a concise way to encode a class of $(P,\omega,\rho)$-partitions with a certain condition on $\rho$. Namely, say that $\rho$ is an \emph{LF-flag} if it satisfies 
\begin{enumerate}[label=(LF)]
\item \label{it:LF} For any  $u,v\in P$, $\rho(u)> \rho(v)$ implies $\omega(u)>\omega(v)$.
\end{enumerate}

\begin{proposition}
\label{prop:Pphi_as_Pomegarho}
Fix a poset $P$.  Then the following hold.
\begin{enumerate}[label=\emph{(\arabic*)}]
\item \label{it:part_un} If $\Phi$ is a labeled flag, then $\mathrm{Part}(P,\Phi)=\mathrm{Part}(P,\omega,\rho)$ for some $\omega,\rho$ satisfying \emph{\ref{it:LF}}.
\item \label{it:part_deux} Conversely, for any $\omega,\rho$ satisfying \emph{\ref{it:LF}}, there exists $\Phi$ such that $\mathrm{Part}(P,\Phi)=\mathrm{Part}(P,\omega,\rho)$.
\end{enumerate}
\end{proposition}
\begin{proof}
 For~\emph{\ref{it:part_un}}, define $\omega=\omega(\Phi)$ as the unique bijection $P\to\{1,\ldots,\#P\}$ that satisfies $\omega(u)<\omega(v)$ if and only if $\Phi(u)<\Phi(v)$. Define also $\rho=\rho(\Phi)$ by $\rho(u)=\mathrm{val}(\Phi(u))$.
 
 We first claim that $\rho$ is LF-flag. Indeed, we have $\rho(u)>\rho(v)$ if and only if $\mrm{val}(\Phi(u))>\mrm{val}(\Phi(v))$, which in turn implies that $\Phi(u)>\Phi(v)$. By the definition of $\omega$ this immediately yields $\omega(u)>\omega(v)$.
 
 We now claim that $\mrm{Part}(P,\Phi)=\mrm{Part}(P,\omega,\rho)$. Say $f\in \mrm{Part}(P,\Phi)$. Then we know that $f(u)\leq \mrm{val}(\Phi(u))=\rho(u)$ for all $u\in P$. Furthermore, $u\prec_P v$ implies $f(u)\geq f(v)$. Finally $u\prec_P v$ and $\Phi(u)>\Phi(v)$ implies $f(u)>f(v)$. Now note that $\Phi(u)>\Phi(v)\Longleftrightarrow \omega(u)>\omega(v)$. It follows that $f\in \mrm{Part}(P,\omega,\rho)$. The inclusion $\mrm{Part}(P,\omega,\rho)\subseteq \mrm{Part}(P,\phi)$ is essentially the same proof.

 For  the converse in \emph{\ref{it:part_deux}}, given $\rho$ an LF-flag define $\Phi(u)=\lf{\rho(u)}{\omega(u)}$.  Like before we need to show that $\mrm{Part}(P,\omega,\rho)=\mrm{Part}(P,\Phi)$, and only the second condition in Definition~\ref{def:P_phi_partition} is not trivial. It follows however readily from the fact that $\Phi(u)>\Phi(v)$ is here equivalent to $\omega(u)>\omega(v)$ because of the condition \ref{it:LF}.
 \end{proof}

\subsection{Comparison with AB-flags} 
The notion of $(P,\Phi)$-partitions extends the notion of AB-flagged partitions defined in the previous section:

\begin{proposition}
Let $\rho$ be an AB-flag for $(P,\omega)$. Define a labeled flag on $P$ by $\Phi(u)=\lf{\rho(u)}{\omega(u)}$. Then 
\[\mathrm{Part}(P,\omega,\rho)=\mathrm{Part}(P,\Phi).\]
\end{proposition}

\begin{proof} Elements of both sets are $P$-partitions that are bounded by $\rho$. It remains to show that the strictness conditions are imposed for the same cover relations. Thus the claim we have to prove is that if $u\prec_P v$, then $\Phi(u)>\Phi(v)$ if and only if $\omega(u)>\omega(v)$.

Pick then $u,v\in P$ such that $u\prec_P v$.
Now $\Phi(u)>\Phi(v)$ is equivalent to either $\rho(u)>\rho(v)$ or ($\rho(u)=\rho(v)$ and  $\omega(u)>\omega(v)$) by definition of the order on $\Paug$. In the latter case, we are already done. In the former case, condition~\ref{it:AB1} implies $\omega(u)>\omega(v)$. 
The converse also holds since we have necessarily $\rho(u)\geq \rho(v)$ by \ref{it:AB2}. 
\end{proof}

The next example shows that in fact, the notion of $(P,\Phi)$-partitions is strictly more general that AB-flagged partitions.

\begin{example} 
Consider $(P,\Phi)$ given by the labeled Hasse diagram on the left of Figure~\ref{fig:counter_ex}. Then $\mathrm{Part}(P,\Phi)$ consists of the five partitions shown on the right.

Let us consider possible choices of $(\omega,\rho)$ such that $\mathrm{Part}(P,\omega,\rho)=\{f_1,f_2,f_3,f_4,f_5\}$ (Note that we know by Proposition~\ref{prop:Pphi_as_Pomegarho} that such a choice is always possible).
   The cover relation $e_1$ is necessarily strict, otherwise the partition $f$ with constant value $1$ would be valid. The other cover relations are necessarily weak since there is a $f_i$ with equality for each of them. This imposes a unique choice for $\omega$, namely $\omega(v_i)=i$ for $i=1,2,3,4$.
   As for $\rho$, we must clearly have $\rho(v_2)=3$ and $\rho(v_3)\geq 3$. Then $\rho(v_1)=1$ is also imposed, since any other choice would make it possible to switch the value of $v_1$ in $f_1$ from $1$ to $2$, resulting in a partition $f$ not in the allowed $f_i$'s.
   
Now assume in addition that $\rho$ is an AB-flag. By \ref{it:AB1} the remaining choices for $\rho$ are $\rho(v_4)=1$ and $\rho(v_3)=3$. This gives us $\rho(v_3)>\rho(v_4)$ while $\omega(v_3)<\omega(v_4)$, contradicting \ref{it:AB2}.  This shows that there are no $\omega,\rho$ with $\rho$ an AB-flag such that $\mathrm{Part}(P,\Phi)=\mathrm{Part}(P,\omega,\rho)$.
\end{example}

\subsection{Slide polynomials}
\label{sub:slide}

Since $(P,\Phi)$-partitions form a certain class of $(P,\omega,\rho)$-partitions, we have the immediate corollary of Theorem~\ref{theo:stanley_restricted}:

\begin{corollary}
\label{coro:stanley_extended}
Let $(P,\Phi)$ be a poset with a labeled flag.  
We have
\[
\mrm{Part}(P,\Phi)=\displaystyle\bigsqcup_{L\in \mrm{Lin}(P)} \mrm{Part}(L,\Phi).
\]
\end{corollary}

We introduce the generating polynomial
\begin{equation}
\label{eq:KPPhi}
K_{(P,\Phi)}=\sum_{f\in\mathrm{Part}(P,\Phi)} \prod_{u\in P} x_{f(u)}.
\end{equation}
Corollary~\ref{coro:stanley_extended} gives:
\begin{equation}
\label{eq:KPPhi_expansion}
K_{(P,\Phi)}=\sum_{L\in\mathrm{Lin(P)}}  K_{(L,\Phi)}.
\end{equation}
 We will now see that all the (nonzero) $K_{(L,\Phi)}$ are {\em slide polynomials} which we introduce next. Write $L=v_1 \prec v_2\prec\cdots \prec v_r$. Any $(L,\Phi)$-partition $f$ can be encoded in a sequence $(i_1,\ldots,i_r)$ with $i_j=f(v_j)$, while $(L,\Phi)$ is encoded in the injective word $W=a_1\dots a_r\in \injwords{\Paug}$ by simply setting $a_i=\Phi(v_i)$.
  The corresponding generating function  $K_{(L,\Phi)}$ is thus given by the explicit series 
\begin{equation}
\label{eq:slide}
\slide(W)\coloneqq K_{(L,\Phi)}=\sum_{\substack{i_1\geq i_{2}\geq\cdots \geq i_r>0 \\ i_j> i_{j+1}\text{ if }a_j>a_{j+1}\\ i_j\leq \mathrm{val}(a_j)}}x_{i_1}x_{i_2}\cdots x_{i_r}.
\end{equation}

\begin{example}
Consider $W=\lf{3}{1}\lf{3}{2}\lf{1}{1}$. Then~\eqref{eq:slide} says 
\[
\slide(W)=\sum_{i_1\geq i_2\geq i_3> 0}x_{i_1}x_{i_2}x_{i_3}
\]
where $3\geq i_1\geq i_2>i_3$ and $i_3\leq 1$. Thus $\slide(W)=x_3^2x_1+x_3x_2x_1+x_2^2x_1$.

Note that replacing the $\lf{3}{2}$ in the middle by any letter in $\Paug$ larger than $\lf{3}{2}$ does not alter $\slide(W)$. 
Finally note that it can be the case that the sum defining $\slide(W)$ is empty. For instance, if $W=\lf{1}{2}\lf{1}{1}$, then $\slide(W)=0$. This `anomaly' will be fixed when we work in the back stable setting.
\end{example}
Given a weak composition $\bfc$ with positive support, consider the word $W_\bfc$ given by ordering the $\lf{i}{j}\in \bZ_\bfc$ with decreasing $i$, and increasing $j$'s for fixed $i$. 

\begin{definition}[\cite{AssBer20}]
The {\em slide polynomial $\slide_\bfc$} is defined as $\slide(W_\bfc)$.
\end{definition}
\begin{proposition}
For any $W\in \injwords{\Paug}$, the polynomial $\slide(W)$ is either zero or is equal to $\slide_\bfc$ for a unique $\bfc$.
\end{proposition}

\begin{proof}
The uniqueness of $\bfc$ holds because slide polynomials form a basis of $\bQ[\alpxp]$. This is the content \cite[Theorem 3.9]{AS}, and follows readily from a triangular change of basis with monomials.

Now given $W$, we must construct a word $W_\bfc$ such that $\slide(W)=\slide(W_\bfc)$. By scanning $W$ from left to right, let us show one can define such a word in an algorithmic fashion. One has to find a standard word $W_\bfc$ with the same descents as $W$, which moreover gives the same upper bounds. We now explain the construction.

We first construct a word $U$ in $\bZ^*$.
If $W$ is empty, then so is $U$. If $W$ has length $1$, then $W=U$.
Otherwise, say $W=w_1\dots w_r$ for $r\geq 2$, with $w_i$ pairwise distinct letters in $\Paug$. Suppose we have scanned letters $w_1$ through $w_k$ for $1\leq k<r$ and constructed a word $u_1\cdots u_k$. 
If $w_{k+1}<w_k$, then $u_{k+1}\coloneqq \min(\mathrm{val}(w_{k+1}),u_k-1)$. If $w_{k+1}>w_k$, then $u_{k+1}\coloneqq u_k$.
Repeating this we get $U=u_1\dots u_r\in \bZ^*$, a word with nonincreasing letters. If any letter $u_i$ is nonpositive, then $\slide(W)=0$.
Otherwise $\stan(U)=W_{\bfc}$ for a unique $\bN$-vector $\bfc$ with positive support, and by induction one checks that $W_{\bfc}$ satisfies $\slide(W)=\slide(W_\bfc)$ as wanted. 
\end{proof}

We refer to the resulting $W_{\bfc}$ as $\reishi(W)$. For instance, let $W=\lf{5}{1}\lf{6}{5}\lf{8}{3}\lf{3}{2}\lf{3}{1}\lf{1}{2}\lf{2}{1}\lf{3}{3}$. Then $U=55532111$ which in turn means that $\reishi(W)=\lf{5}{1}\lf{5}{2}\lf{5}{3}\lf{3}{1}\lf{2}{1}\lf{1}{1}\lf{1}{2}\lf{1}{3}$. This equals $W_{\bfc}$ for $\bfc=(\dots,0|3,1,1,0,3,0,\ldots)$. 
Note that if the last letter in $W$ is replaced by say $\lf{2}{3}$, then one gets $\slide(W)=0$. This is because $W_{\bfc}=\reishi(W)$ satisfies $\supp(\bfc)\nsubseteq\bZ_+$. 

 The combinatorics  explained in the proof are already presented by Reiner and Shimozono in \cite[Lemma 8]{ReiShi95}: the minor differences are that the authors in \emph{loc. cit.} work from right to left, and work with $\bP$ instead of~$\Paug$.

\subsection{Forest polynomials}
We close this section  by introducing a novel family of polynomials that will form the core of a forthcoming work that links the quotient of $\bQ[x_1,\dots,x_n]$ by the ideal of positive degree quasisymmetric polynomials in $x_1,\dots,x_n$ to the Schubert class expansion of the cohomology class of the permutahedral variety $\mathrm{Perm}_n$.
We lay the foundations for that work here and postpone the discussion of their combinatorics, as well as connections to Schubert polynomials, for the future.

An \emph{indexed tree} $T_I\coloneqq (T,I)$ consists of an interval $I=[a,b]\subset \bZ_+$ with $a<b$ and a rooted plane binary tree $T$ on $b-a$ nodes. 
We can pictorially think of $T_I$ as the tree obtained by completing $T$ by adding $b-a+1$ leaves and then labeling the leaves from left to right with the integers $a$ through $b$.
An \emph{indexed forest} $F$ is a collection $(T_I)$ of indexed trees where the intervals are all disjoint.

It will be convenient to visualize an indexed forest as a collection of complete binary trees supported on the integer lattice $\bZ$ with the leaves of each tree corresponding to the appropriate interval in $\bZ$. 
Let us denote the internal nodes (i.e. everything but the leaves) in $F$ by $\mathrm{IN}(F)$ and the latter's cardinality by $|F|$.
See Figure~\ref{fig:indexed_plus_flag} for a complete indexed forest (ignoring the labels in red for now) with $|F|=6$. 
It comprises three rooted plane trees supported on the intervals $[2,5]$, $[7,8]$, and $[11,13]$ from left to right.

\begin{figure}[!ht]
\includegraphics[scale=0.7]{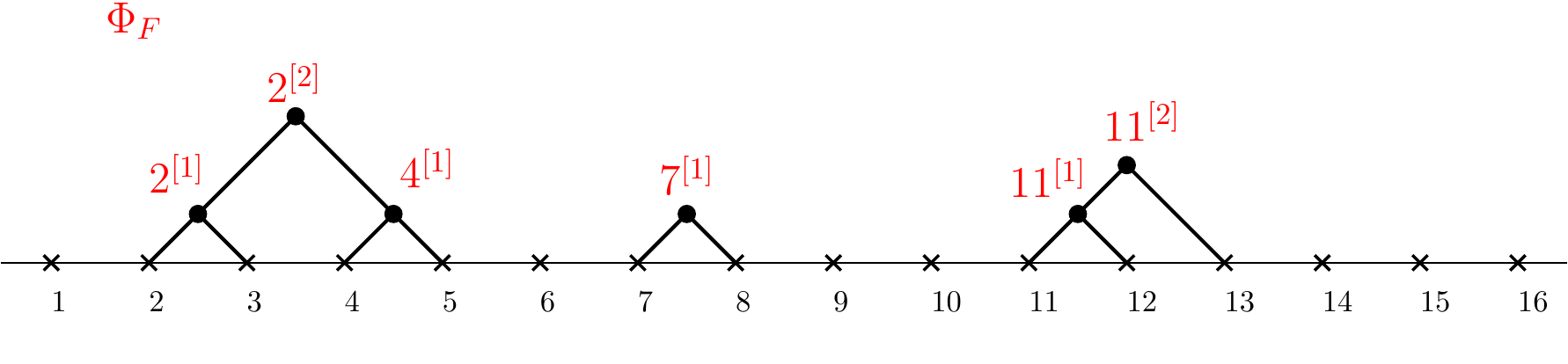}
\caption{An indexed forest with flag}
\label{fig:indexed_plus_flag}
\end{figure}

We may treat an indexed forest $F$ as a Hasse diagram of a poset (also called $F$) with underlying set $\mathrm{IN}(F)$. 
\emph{What roles do the intervals play?} They determine a labeled flag $\Phi_F$ as follows: for each left leaf $lf$, let $m$ be its label and $v_1,\dots,v_k$ be the inner nodes on the left branch ending at $lf$, from bottom to top. Then define $\Phi_F(v_i)\coloneqq \lf{m}{i}$ for $i=1,\ldots,k$. 
 In Figure~\ref{fig:indexed_plus_flag} the left leaves have labels $2,4,7$ and $11$, completely determining $\Phi_F$ as depicted in red.

\begin{definition}
The \emph{forest polynomial} $\forest_F\in\bQ[\alpxp]$ is defined as
\[
\forest_F=K_{(F,\Phi_F)}.
\]
Explicitly, $\forest_F$ is the sum of monomials $\prod_{v\in \mathrm{IN}(F)}x_{f(v)}$ over all labelings $f:\mathrm{IN}(F)\to\bP$ satisfying $f(v)\leq \mathrm{val}(\Phi_F(v))$ for all $v$, and that are weakly decreasing down left edges and strictly increasing down right edges.\end{definition}

Thus $\forest_F$ is the sum of monomial over `semi-standard' increasing labelings of the internal nodes of the indexed forest $F$.
For example, given the indexed forest $F$ in Figure~\ref{fig:indexed_forest_0201}
\begin{figure}
\includegraphics[scale=0.8]{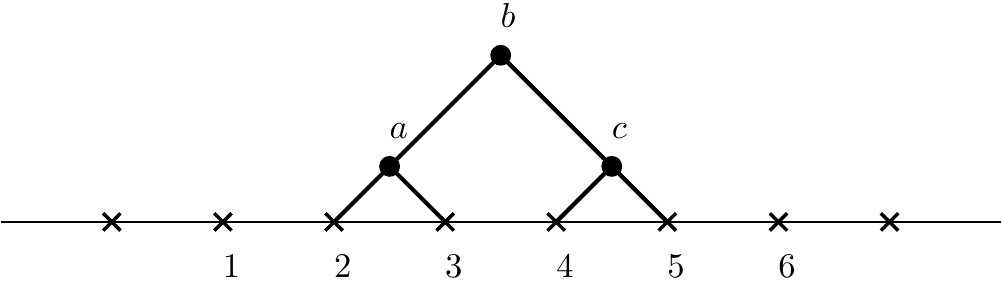}
\caption{An indexed forest $F$ with $c(F)=(0,2,0,1)$}
\label{fig:indexed_forest_0201}
\end{figure} 
one has that
\begin{align*}
\forest_F=\sum_{\substack{\textcolor{red}{2}\geq a\geq b\\ \textcolor{red}{4}\geq c>b}}x_ax_bx_c&=
x_{2}^{2} x_{4} + x_{1} x_{2} x_{4} + x_{1}^{2} x_{4} + x_{2}^{2} x_{3} + x_{1} x_{2} x_{3} + x_{1}^{2} x_{3}  + x_{1}^{2} x_{2}+ x_{1} x_{2}^{2}.
\end{align*}

We can obtain the expansion of $\forest_F$ into slide polynomials thanks to~\eqref{eq:KPPhi_expansion}. Note that a linear extension of the poset $F$ corresponds to a \emph{decreasing forest}. Let $\mrm{Dec}(F)$ be the set of such decreasing labelings, and for any such labeling $\ell$ we note $W(\ell)$ the injective word read from the flag $\Phi_F$. We thus get
\begin{equation}
\forest_F=\sum_{\ell\in \mrm{Dec}(F)} \slide(W(\ell)).
\end{equation}
Returning to our running example we get two decreasing labelings for the $F$ in Figure~\ref{fig:indexed_forest_0201}, which in turn imply the slide expansion
\[
\forest_F=\slide({\lf{2}{1}\lf{4}{1}\lf{2}{2}})+\slide({\lf{4}{1}\lf{2}{1}\lf{2}{2}})=\slide_{12}+\slide_{0201}.
\]

By appealing to a simple bijection between $\bN$-vectors and indexed forests we can show that the set of forest polynomials is a basis for the polynomial ring $\bZ[\alpxp]$.
Indeed, let $c(F)$ be the weak composition given by $c_i(F)$ being the number of nodes on the left branch leading to the leaf labeled $i$. 
It is easily seen that the correspondence $F\mapsto c(F)$ is a bijection between indexed forests and $\bN$-vectors supported on $\bZ_+$.
In fact $\alpx^{c(F)}$ is the revlex leading monomial in $\forest_F$, which in turn suffices to show the following:
\begin{theorem}
 The set of forest polynomials forms a basis for $\bZ[\alpxp]$. 
\end{theorem}

\section{Back stable version}
\label{sub:backstable}

We shall see that the theory of $(P,\Phi)$-partitions (and $(P,\omega,\rho)$-partitions more generally) becomes nicer when we allow $P$-partitions to take nonpositive values. In the expansion~\eqref{eq:KPPhi_expansion} of $K_{(P,\Phi)}$, some terms on the right-hand side can be zero. In fact $K_{(P,\Phi)}$ itself may be zero, and this phenomenon leads to some technical investigation in~\cite{AssBer20}. Removing the lower bound for $(P,\Phi)$-partitions ensures that no term will cancel, leading to a more pleasant theory.

 The corresponding generating functions then belong to a certain class of series.
We describe some structural theory of these series, borrowing notation and drawing motivation  from \cite{Lam18}. We also note that some of the series we consider arose naturally in previous work; see \cite[Proposition 8.9]{NT21} and \cite[Section 3]{TWZ}. 

\subsection{Back stable $(P,\Phi)$-partitions}

We now consider $(P,\Phi)$-partitions $f$ with the added liberty that $f$ can now take values in $\bZ$ instead of $\bP$.

\begin{definition}[Back stable $(P,\Phi)$-partitions]
\label{def:backstabe_P_phi_partition}
Given a poset $P$, let $\Phi$ be an injective function $\Phi:P\to \Zaug$. A \emph{back stable} $(P,\Phi)$-partition is a function $f:P\to \bZ$ such that for any $u,v\in P$:
\begin{itemize}
\item $f(u)\geq f(v)$ if $u\prec_P v$;
\item $f(u)>f(v)$ if $u\prec_P v$ and $\Phi(u)>\Phi(v)$;
\item $f(u)\leq \mathrm{val}(\Phi(u))$.
\end{itemize} 
\end{definition}
Denote by $\overleftarrow{K}_{(P,\Phi)}$ the corresponding generating function. Note that this is now a homogeneous series of degree $\#P$ in $(x_i)_{i\in\bZ}$ such that only a finite number of $x_i$ for $i>0$ occur.
Then Stanley's fundamental theorem, in the form of Corollary~\ref{coro:stanley_extended}, holds with no change.
For $L$ a chain, identified as a word $W\in \injwords{\Zaug}$ as in Section~\ref{sub:slide}, let $\backslide(W)=\overleftarrow{K}_{(L,\Phi)}$. Then we have
\begin{align}
\label{eq:poset_quasi_back}
\overleftarrow{K}_{(P,\Phi)}=\sum_{{L\in \mathrm{Lin}(P)}} \backslide(W).
\end{align}

Now the series $\overleftarrow{K}_{(P,\Phi)}$ is never zero. In particular,  all the series $\backslide(W)$ occur in the previous inequality, which is an actual summation over all linear extensions.
Another advantage is that $\overleftarrow{K}_{(P,\Phi)}$ lives inside the vector space of back stable quasisymmetric functions, introduced in the next subsection.
This space comes equipped with several linear maps, which allow for recovering both the classical story of $(P,\omega)$-partitions and quasisymmetric functions as well as the polynomial story involving $(P,\Phi)$-partitions developed here.

\subsection{Back stable quasisymmetric functions}
\label{sub:backstable_quasisymmetric}

We work with the set of variables $\alpx\coloneqq \{x_i\suchthat i\in\bZ\}$.
Given $b\in \bZ$, define $\alpx_{\leq b}\coloneqq \{x_i\suchthat i\leq b\}$. In the case $b=0$, we set $\alpxm\coloneqq \alpx_{\leq 0}$. 
Let $\alpxp\coloneqq\{x_i|i\geq 1\}$.
Let $\qsym(\alpx_{\leq b})$ be the $\bQ$-algebra
of quasisymmetric functions in the ordered alphabet $\alpx_{\leq b}$. 
The algebra $\qsym(\alpxm)$ will be denoted by $\qsym$.

Let $R \subseteq \bQ[\![\alpx]\!]$ denote the set of bounded degree formal power series $f$ in  $\alpx$ with the property that there exists an $N\in \bZ$ such that no $x_i$ appears in $f$ for $i>N$.

\begin{definition}
\label{def:backstablequasisymmetric}
Let $f\in R$. We say that $f$ is \emph{back quasisymmetric} if there exists a $b\in \bZ$ such that for any sequence of positive integers $(a_1,\dots,a_k)$ and \emph{any} monomial $\mathsf{m}$ in $\alpx_{>b}$, the coefficient of $x_{i_1}^{a_1}\cdots x_{i_k}^{a_k}\mathsf{m}$ in $f$ equals that of $x_{j_1}^{a_1}\cdots x_{j_k}^{a_k}\mathsf{m}$ whenever $i_1<\cdots <i_k\leq b$ and $j_1<\cdots <j_k\leq b$. 
\end{definition}
Equivalently, $f$ is back quasisymmetric if there exists a $b\in \bZ$ such that $f\in \qsym(\alpx_{\leq b})\otimes \bQ[x_{b+1},x_{b+2},\dots]$.

 \begin{remark}
 For the reader looking for something in the spirit of \cite{Lam18}, the following alternative definition should suffice.
 Call $f\in R$ back quasisymmetric if there exists a $b\in \bZ$ such that $\boldsymbol{\sigma}_i(f)=f$ for all $i\leq b$ where $\boldsymbol{\sigma}_i$ acts via Hivert's \emph{quasisymmetrizing action} \cite[Section 3]{Hiv}.
\end{remark}

Let $\backquasi$ denote the space of back quasisymmetric functions. 
Clearly, the space $\overleftarrow{R}$  of back symmetric functions \cite{Lam18} is a subset of $\backquasi$.
We have the following analogue of~\cite[Proposition 3.1]{Lam18}.
\begin{proposition}
\label{prop:3.1LLS}
We have that
\begin{align*}
\backquasi = \qsym \otimes \bQ[\alpx].
\end{align*}
\end{proposition}
\begin{proof}
Both $\qsym$ and $\bQ[\alpx]$ are contained in $\backquasi$. Let us show that their abstract tensor product naturally embeds naturally in $\backquasi$.
We need to prove that for any linear relation in $\backquasi$
\begin{align}
\label{eq:linear_relation_backquasi}
\sum_{\alpha,c}u_{\alpha,{\bfc}}F_\alpha(\alpxm)\alpx^{\bfc}=0,
\end{align}
the coefficients $u_{\alpha,{\bfc}}\in \bQ$ are all zero. Here $\alpha$ and $\bfc$ run over compositions and $\bN$-vectors respectively. Now fix an $\bN$-vector $\mathbf{d}$. After dividing the relation in~\eqref{eq:linear_relation_backquasi} by $\alpx^{\mathbf{d}}$, we get
\begin{align}
\label{eq:take_limits}
\sum_{ \alpha}u_{\alpha,\mathbf{d}}F_\alpha(\alpxm)=-\sum_{\alpha,\bfc\neq \mathbf{d}}u_{\alpha,{\bfc}}F_\alpha(\alpxm)\alpx^{\bfc-\mathbf{d}}.
\end{align}
Now we apply the shift $x_i\mapsto x_{i+b}$ for $b$ a nonnegative integer, and then set $x_i=0$ for $i\leq0$. 
The expression $F_\alpha(\alpxm)\alpx^{\bfc-\mathbf{d}}$ becomes
$F_\alpha(x_1,\ldots,x_b)\prod_{i}x_{i+b}^{{\bfc}_i-\mathbf{d}_i}$. 
As $b$ goes to infinity, this quantity goes to zero since $\bfc\neq \mathbf{d}$. The limit here is the usual one for series, where for any fixed monomial the coefficients eventually~stabilize.
 
The relation in~\eqref{eq:take_limits} in the limit then gives 
\begin{align}
\sum_{ \alpha}u_{\alpha,\mathbf{d}}F_\alpha(\alpxp)=0.
\end{align} 
We can now conclude that $u_{\alpha,\mathbf{d}}=0$ for all $\alpha$ by linear independence of fundamental quasisymmetric functions. Since $\mathbf{d}$ was arbitrary, we get the desired result.

 Having thus shown  $\backquasi \supseteq \qsym \otimes \bQ[\alpx]$, we proceed to prove the reverse inclusion.
 Let $f\in \backquasi$, so that there exists $b\in \bZ$ such that $f\in \qsym(\alpx_{\leq b})\otimes \bQ[x_{b+1},x_{b+2},\dots]$.
By linearity, it is enough to assume that $f=F_{\alpha}(\alpx_{\leq b})P$ with $P\in \bQ[x_{b+1},x_{b+2},\ldots]$. Say $b\geq 1$. Write the fundamental quasisymmetric function $F_{\alpha}(\alpx_{\leq b})$ as $F_{\alpha}(\alpxm + x_1+\cdots+x_{b})$. We know that this expands as 
\begin{align}
\label{eq:option_1}
F_{\alpha}(\alpx_{\leq b})=\sum_{\beta \odot \gamma=\alpha \text{ or } \beta\cdot\gamma=\alpha} F_{\beta}(\alpxm)F_{\gamma}(x_1,\dots,x_{b}).
\end{align}
Now suppose $b\leq -1$. Then $
F_{\alpha}(\alpx_{\leq b})=F_{\alpha}(\alpxm - x_0-\cdots-x_{b+1})$.
This time we know that
\begin{align}
\label{eq:option_2}
F_{\alpha}(\alpx_{\leq b})=\sum_{\beta \odot \gamma=\alpha \text{ or } \beta\cdot\gamma=\alpha} (-1)^{|\gamma|}F_{\beta}(\alpxm)F_{\gamma^t}(x_{b+1},\dots,x_{0}).
\end{align}
In both cases, this shows that $f=F_{\alpha}(\alpx_{\leq b})P$ is an element of $\qsym\otimes\bQ[\alpx]$ as wanted.
\end{proof}

\subsection{Some maps defined on back stable quasisymmetric functions }
As in \cite[Section 3.4]{Lam18}, we consider the evaluation map $\eta_0:\bQ[\alpx]\to\bQ$  obtained by setting $x_i=0$ for all $i\in \bZ$. 
In other words, it picks the constant term in a polynomial. 
It induces the map $1\otimes \eta_0$ on $\qsym \otimes \bQ[\alpx]$: it essentially picks out the term in $\qsym$ and forgets the polynomial part. 
Following Lam--Lee--Shimozono, we will abuse notation and refer to this induced map on $\backquasi$ by $\eta_0$ as well.

Let $\gamma:\backquasi\to \backquasi$ be the map shifting variables $x_i\mapsto x_{i+1}$ for $i\in\bZ$ \cite[Section 3.3]{Lam18}.
Finally let $\pi_+:\backquasi\to \bQ[\alpxp]$ be the \emph{truncation map} obtained by setting $x_i=0$ for $i\leq 0$; see proof of \cite[Proposition 3.18]{Lam18}. Note that all of these maps are algebra morphisms, and the previous two  were already employed in the proof of Proposition~\ref{prop:3.1LLS}.

\begin{proposition}
For any $f\in \backquasi$, \[\eta_0(f)(\alpxp)=\lim_{b\to\infty} \pi_+(\gamma^b(f)).\]
\end{proposition}
Note that $\eta_0(f)$ lives in $\qsym(\alpxm)$. The notation $\eta_0(f)(\alpxp)$ means that we now write it in the variables $\alpxp$ using the natural isomorphism between $\qsym(\alpxm)$ and $\qsym(\alpxp)$. The limit is the usual one for series, already used in the proof of Proposition~\ref{prop:3.1LLS}: the coefficients of any fixed monomial eventually~stabilize.

\begin{proof}
By linearity it is enough to prove it for $f=g\alpx^\bfc$ with $g\in\qsym$, with $\alpx^\bfc=x_{i_1}^{a_1}\cdots x_{i_k}^{a_k}$. Then $\eta_0(f)=g$ if $\alpx^\bfc=1$ and $0$ otherwise. On the other hand, for $b$ large enough $\pi^+(\gamma^b(f))=g(x_1,\ldots,x_b)x_{i_1+b}^{a_1}\cdots x_{i_k+b}^{a_k}$. This has limit $0$ if any of the $a_i$ is nonzero, and $g(\alpxp)$  otherwise as wanted.
\end{proof}

\subsection{Back stable slides}
We now discuss the \emph{back stable slides} $\backslide_{\bfc}\coloneqq \backslide(W_{\bfc})$ beginning with establishing that they belong in $\backquasi$.
In particular, this would show via~\eqref{eq:poset_quasi_back} that $\overleftarrow{K}_{(P,\Phi)}\in\backquasi$ as well.
To this end, we will need a result from \cite{TWZ} that we now recall for the reader's convenience.

Let $\bfc$ be an $\bN$-vector. We let $\flatten{\bfc}$ denote the sequence formed by the positive entries in $\bfc$.
Call a decomposition $\bfc=\bf{d}+\bf{e}$ where addition is component-wise \emph{good} if either $\flatten{\bfc}=\flatten{\bf{d}}\cdot \flatten{\bf{e}}$ or $\flatten{\bfc}=\flatten{\bf{d}}\odot\flatten{\bf{e}}$ holds.
Then \cite[Lemma 3.5]{TWZ} in a special case states that:
\begin{lemma}
\label{lem:twz}
Let $\bfc$ be an $\bN$-vector such that $\supp(\bfc)\subseteq \bZ_{>0}$. Then $\backslide_{\bfc}$ has the following expansion in $\qsym\otimes \bQ[\alpxp]$:
\[
\backslide_{\bfc}=\sum_{\text{good }\bfc=\bf{d}+\bf{e}} F_{\flatten{\bf{d}}}(\alpxm) \slide_{\bf{e}}(\alpxp).
\]
\end{lemma}

\begin{example}
Let $\bfc=(0,2,0,2,0,0,\ldots)$.  Then $\flatten{\bfc}=(2,2)$ and one can easily check that we have the following five decompositions for $(2,2)$: $\varnothing \cdot (2,2)$, $(1)\odot (1,2)$, $(2)\cdot (2)$, $(2,1)\odot (1)$, $(2,2)\cdot \varnothing$. These in turn translate to five good decompositions and we obtain
\[
\backslide_{0202}=\slide_{0202}(\alpxp)+F_{1}(\alpxm)\slide_{0102}(\alpxp)+F_{2}(\alpxm)\slide_{0002}(\alpxp)+F_{21}(\alpxm)\slide_{0001}(\alpxp)+F_{22}(\alpxm).
\]
\end{example}

\begin{proposition}
\label{prop:in_back_quasi}
For any $\bN$-vector $\bfc$, we have $\backslide_{\bfc}\in \backquasi$.
As a consequence, $\overleftarrow{K}_{(P,\Phi)}\in \backquasi$ for any poset $P$ with labeled flag $\Phi$.
\end{proposition}
\begin{proof}
	It suffices to assume $\supp(\bfc)\subset \bZ_{>0}$. Indeed, abuse notation and define $\gamma(\bfc)$ to be the $\bN$-vector obtained by shifting $\bfc$ once to the right. Then it is clear that for $i\in \bZ$ that
	\[
	\gamma^{i}\left(\backslide_{\bfc} \right)=\backslide_{\gamma^i(\bfc)}.
	\]
	Now note that $\backquasi$ is closed under shifting.
	
	When $\supp(\bfc)\subset\bZ_{>0}$, then Lemma~\ref{lem:twz} immediately implies that $\backslide_{\bfc}\in \backquasi$. As for showing $\overleftarrow{K}_{(P,\Phi)}\in \backquasi$, note that it is a sum of back stable slides as in~\eqref{eq:poset_quasi_back}.
\end{proof}

The next result describes the distinguished role played by the back stable slides in $\backquasi$. It is the analogue of \cite[Theorem 3.5]{Lam18}.
\begin{theorem}
\label{theo:backslide_basis}
The back stable slides $\backslide_{\bfc}$ for $\bfc$ an $\bN$-vector form a $\bQ$-basis of $\backquasi$.
\end{theorem}
\begin{proof}
The linear independence holds for the same reason as in \cite{Lam18}. Indeed the revlex leading monomial in $\backslide_{\bfc}$ is  $\alpx^{\bfc}$.
We now show that the $\backslide_{\bfc}$ are spanning. 

Pick $f\in \backquasi$. 
Without loss of generality assume that $f\in \qsym\otimes\bQ[\alpx_+]$: indeed, like before, this follows from the fact that $\backquasi$ is closed under shifting.

We know that slide polynomials form a basis for $\bQ[\alpxp]$, so  we may assume that $f$ has the form $f=F_{\alpha}(\alpxm)\slide_{(c_1,\dots,c_k)}$. If $c_i=0$ for all $1\leq i\leq k$, then we are done. 
Otherwise let ${\bf d}=(\dots,0,\alpha|c_1,\dots,c_k,0,\dots)$. By \cite[Lemma 3.5]{TWZ} (essentially the statement in Lemma~\ref{lem:twz}) the difference $f-\backslide_{\bf{d}}$ is a sum of $F_{\beta}(\alpxm)\slide_{(a_1,\dots,a_k)}$ where $a_1+\dots+a_k<c_1+\dots+c_k$.
Induction implies the claim.  
\end{proof}

\begin{remark}As the reader may expect at this stage, by allowing our indexed forests to be supported on $\bZ$ rather than $\bZ_+$, we may easily define \emph{back stable forest polynomials}.
The resulting family of polynomials $\backforest_F\in \backquasi$ then expands as a sum of back stable slide polynomials, one for each element in $\mrm{Dec}(F)$.
\end{remark}

The lemma next saying that slides and fundamental quasisymmetric functions simultaneously inhabit $\backslide_{\bfc}$ is straightforward.
\begin{lemma}
We have $\eta_0(\backslide_{\bfc})=F_{\flatten{\bfc}}(\alpxm).$ Additionally
\[
\pi_+(\backslide_{\bfc})=\left\lbrace\begin{array}{ll}\slide_{\bfc} & \supp(\bfc)\subset \bZ_+\\
0 & \text{otherwise.}\end{array}\right.
\]
\end{lemma}

Our next result follows from Stanley's theory again. It gives a \emph{shuffle rule} for multiplying back stable slide polynomials.
Let $\bfc$ and $\bf{d}$ be $\bN$-vectors. We have $W_{\bfc}$ and $W_{\bf{d}}$ defined as usual. Define the set of shuffles $W_{\bfc}\shuffle W_{\bf{d}}$ where we replace every instance of $\lf{i}{j}$ in $W_{\bf{d}}$ by $\lf{i}{j+c_i}$. This shift ensures, amongst other things, that the set of shuffles comprises injective words.
\begin{proposition}
\label{prop:multiplying_slides}
Given $\bN$-vectors $\bfc$ and $\bf{d}$ we have
\[
\backslide_{\bfc}\cdot \backslide_{\bf{d}}=\sum_{L\in W_{\bfc}\shuffle W_{\bf{d}}} \backslide(L).
\]
\end{proposition}
Note that each summand on the right-hand side is a back stable slide polynomial. Hitting the expansion with $\eta_0$ gives us the usual shuffle product for $F_{\flatten{\bfc}}\cdot F_{\flatten{\bf{d}}}$, while applying $\pi_+$ recovers the rule for slides.

\begin{example}
Let $\bfc=(0,1,0,2)$ and ${\bf d}=(0,1)$. Then $W_{\bfc}=\tcb{\lf{4}{1}\lf{4}{2}\lf{2}{1}}$ and $W_{\bf{d}}=\tcr{\lf{2}{1}}$. We have that
\[
W_{\bfc}\shuffle W_{\bf{d}}=\{\tcb{\lf{4}{1}\lf{4}{2}\lf{2}{1}}\tcr{\lf{2}{2}}, \tcb{\lf{4}{1}\lf{4}{2}}\tcr{\lf{2}{2}}\tcb{\lf{2}{1}}, \tcb{\lf{4}{1}}\tcr{\lf{2}{2}}\tcb{\lf{4}{2}\lf{2}{1}},\tcr{\lf{2}{2}}\tcb{\lf{4}{1}\lf{4}{2}\lf{2}{1}} \}
\]
Consider the $L\in W_{\bfc}\shuffle W_{\bf{d}}$ from left to right. The corresponding  $\backslide(L)$ are $\backslide_{(0,2,0,2)}$, $\backslide_{(1,1,0,2)}$, $\backslide_{(1,2,0,1)}$, and $\backslide_{(1,3)}$ respectively.
So we infer that 
\[
\backslide_{0102}\cdot \backslide_{01}=\backslide_{0202}+\backslide_{1102}+\backslide_{1201}+\backslide_{13}.
\]
\end{example} 

\section{The inverse slide Kostka matrix}

We showed in Theorem~\ref{theo:backslide_basis} that back stable slide polynomials form a basis of $\backquasi$. In view of Proposition~\ref{prop:3.1LLS}, another basis is given by the $F_\alpha(\alpxm)\alpx^{\bfc}$ for any composition $\alpha$ and $\bN$-vector $\bfc$. In this section we give an explicit change of basis from the first basis to the second. Since we know how to multiply back stable slide polynomials by Proposition $\ref{prop:multiplying_slides}$, it is enough to treat separately the case of $F_\alpha(\alpxm)$ and $\alpx^{\bfc}$. Now $F_\alpha(\alpxm)$ is already equal to the back stable slide $\backslide_\bfc$ with $\bfc=(\ldots,0,{\alpha}|0,0,\ldots)$, so it remains to express a monomial  $\alpx^{\bfc}$ in terms of back stable slide polynomials.

Recall that we define $\backslide_{\bfc}$ as $\backslide(W_{\bfc})$, where $W_\bfc$ is a special word attached to $\bfc$. These words are precisely the standardizations of \emph{nonincreasing} words in $\bZ^*$, that is words where letters decrease weakly from left to right. We will use letters $C,D,E$ for nonincreasing words, to recall that they correspond bijectively  to  $\bN$-vectors $\bfc,{\bf d}, {\bf e}$, but we will not use the latter. We will also write simply $\slide(C)$ and $\backslide(C)$ as no confusion is caused.
 
Let $\mathcal{P}$ denote the set of all nonincreasing words.
Fix $C=C_1C_2\cdots C_m \in \mathcal{P}$ with $C_i\in\bZ$.
By grouping equal values one may write $C=M_1^{m_1}M_2^{m_2}\cdots M_t^{m_t}$ where $m_i>0$ and $M_i>M_{i+1}$ for $1\leq i<t$. 
 Fix $i\in{1,\ldots,t}$, and let $x_0=M_i$ by convention. Now let 
 \begin{align}
 B_i=\{x_1\cdots x_{m_i}\in \mathcal{P}\suchthat x_{j+1}\in \{x_{j},x_j-1\}\text{ for } 0\leq j\leq m_i-1 \text{, and } x_{m_i}>M_{i+1}\}.
 \label{eq:def_B_i}
 \end{align}
 This given define the following set that is crucial for us:
 \[
 \mathbb{B}_C=\{X^1\cdots X^t\suchthat X^i\in B_i\text{ for } 1\leq i\leq t\}.
 \]
Note that $\mathbb{B}_C\subset \mathcal{P}$ by construction. 
Elements $D\in \mathbb{B}_{C}$ are completely characterized by the set $S_C(D)$ of indices $j$ such that $x_{j+1}=x_j-1$ in~\eqref{eq:def_B_i}.
In particular, this allows us to identify $\mathbb{B}_C$ with a distinguished subset of the set of sequences $(S_1,\dots,S_t)$ where each $S_i\subseteq \{0,\dots,m_i-1\}$. Such sequences in turn may naturally be identified as subsets of $\{0,1,\dots,m_1+\cdots+m_t-1\}$.
In fact, as we shall soon see, this aforementioned association has even further structure; we have an isomorphism between appropriate posets.

 \begin{example}\label{ex:B_C}
 Let $C=442=4^2 \, 2^1$. 
 Then $B_1$ and $B_2$ equal  $\{44,43,33\}$  and $\{2,1\}$ respectively.
 We thus have
 \[
 \mathbb{B}_C=\{442,332,432,441,331,431\}.
 \]
 Recording the $S_C(D)$ as $D$ varies over $\mathbb{B}_C$ gives us the following subsets of $\{0,1,2\}$:
 \[
 \{\emptyset,\{0\},\{1\},\{2\},\{0,2\},\{1,2\}\}.
 \]
 Note that these subsets give a lower order ideal in the Boolean lattice on subsets of $\{0,1,2\}$.
 \end{example}

We are now ready to state the main result in this subsection. It expresses the monomial $\alpx(C)\coloneqq x_{C_1}\cdots x_{C_m}$ in terms of back stable slides. In particular, the expansion is \emph{signed multiplicity free}.
  
\begin{theorem}
\label{thm:monomials_into_slides}
The back stable slide expansion of $\alpx(C)$  is given by
\begin{equation}
\label{eq:monomials_into_backslides}
\alpx(C)=\sum_{D\in \mathbb{B}_{C}} (-1)^{|S_C(D)|}\, \backslide(D).
\end{equation}
\end{theorem}
We illustrate the theorem with an example.
\begin{example}
Let $C=442$, so that $C=W_\bfc$ with $\bfc=(0,1,0,2,0,\dots)$. 
We already computed $\mathbb{B}_C$ in Example~\ref{ex:B_C}.
Theorem~\ref{thm:monomials_into_slides} then states:
\begin{align*}
x_2x_4^2=\backslide(442)-\backslide(332)-\backslide(432)-\backslide(441)+\backslide(331)+\backslide(431).
\end{align*}
\end{example}

We need some preparation before presenting the proof of Theorem~\ref{thm:monomials_into_slides}.
We will appeal to poset-theoretic terminology freely; the reader is referred to \cite[Chapter 3]{St97} for any undefined jargon.

\begin{definition}
Let $\mathcal{P}_m\subset \mathcal{P}$ be the set of all nonincreasing words of length $m$. For $C,D\in\mathcal{P}$, define $D\leq_{\pa} C$ if and only if $D_i\leq C_i$ for all $i$ and $D_i>D_{i+1}$ whenever $C_i>C_{i+1}$. 
\end{definition}

This makes it clear that $(\mathcal{P}_m,\leq_{\pa})$ is a poset. In fact it is \emph{locally finite}, thus we have the existence of a M\"obius function $\mu$. We recall that it is defined on all $(D,C)$ with $D\leq_{\pa}C$ by $\mu(C,C)=1$ for all $C$, and whenever $D<_{\pa}C$,
\begin{equation}
\label{eq:moebius}
\sum_{D\leq_{\pa}E\leq_{\pa} C}\mu(E,C)=0.
\end{equation}
We illustrate a convex subset of $\mathcal{P}_3$ in Figure~\ref{fig:P3}.

\begin{figure}
\includegraphics[]{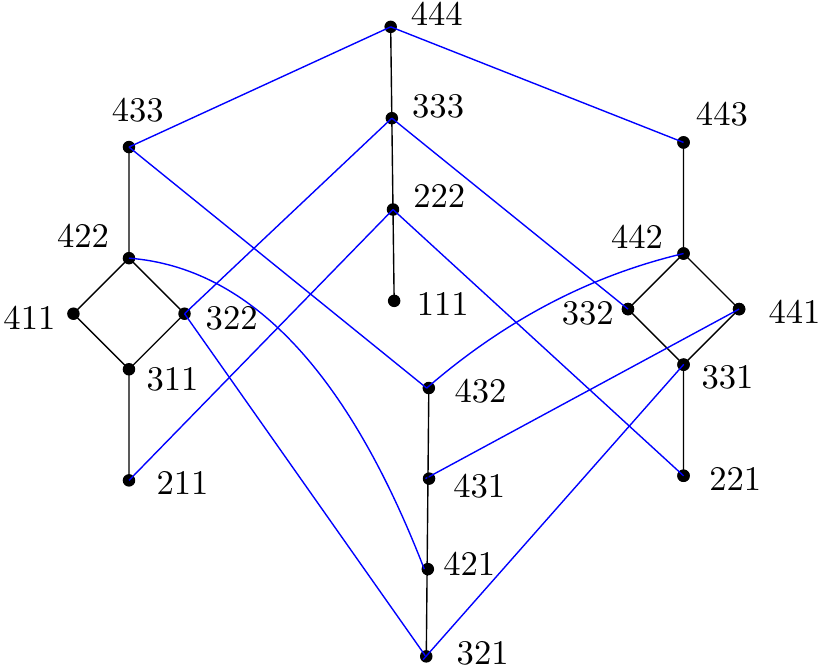}
\caption{\label{fig:P3} Hasse diagram of the convex subset of $\mathcal{P}_3$ formed by all words below $444$ and with positive letters.}
\end{figure}

Toward describing the M\"obius function of $\mathcal{P}_m$, we will first show that it is a lattice, that is, any two elements have a \emph{join} (least upper bound) and a \emph{meet} (greatest lower bound). It is helpful to consider an example of a join of two elements, as that will guide the construction that follows.

Consider $C=555322$ and $D=664421$. Any common upper bound must be component-wise greater than both $C$ and $D$. This leads one to propose $E'=665422$. Now observe that $E'_2>E'_3$ yet $C_2=C_3$, and $E'_3>E'_4$ but $D_3=D_4$. Therefore, whilst $E'$ is component-wise greater than both $C$ and $D$, we do not have $C\leq_{\pa} E'$ and $D\leq_{\pa} E'$. There is an easy fix-- increment $E'_3$ and  $E'_4$ appropriately so that the resulting word does not have strict descents in positions $2$ and $3$. Indeed, check that $E=666622$ meets these criteria and therefore $E=C\vee D$.

\begin{proposition}
\label{prop:lattice}
For any $m\geq 1$, $\mathcal{P}_m$ is a lattice.
\end{proposition}
\begin{proof}
Let $C,D$ be two elements of $\mathcal{P}_m$. We will first prove the existence of their join by generalizing the construction above.
 Define $E'\in \mathcal{P}_m$ by $E'_i=\max(C_i,D_i)$. 
Let $J$ be the set consisting of $1$  and all $i\in \{2,\ldots,m\}$ such that $C_{i-1}>C_{i}$ and $D_{i-1}>D_{i}$. So, for instance, in the example leading to this proposition, we would have $J=\{1,5\}$. 
Now define $E\in\mathcal{P}_m$ by setting $E_i=E'_{j_k}$ where $j_k$ the the maximal element of $J$ such that $j_k\leq i$. 

We claim that $E$ is the join of $C$ and $D$. Let us first check that $E$ is an upper bound for $C$ and $D$.
Indeed we have $E_i\geq E'_i$ and $E'_i\geq C_i,D_i$, so $E$ is component-wise greater than $C$ and $D$. Second, $E_{i-1}>E_{i}$ implies that $i\in J$ by construction, which by definition entails $C_{i-1}>C_{i}$ and $D_{i-1}>D_{i}$. Thus we infer that $E\geq_{\pa} C,D$.

Now let $G$ satisfy $G\geq_{\pa} C,D$. One thus has clearly that $G_i\geq E'_i$ for all $i$, and since $i\notin J$ implies that $G_{i-1}=G_i$, it follows that in fact $G_i\geq E_i$ for all $i$. Now if $G_{i-1}>G_{i}$, one necessarily  has $i\in J$. We have thus $E_i=E'_i=\max(C_i,D_i)$ and $E_{i-1}\geq E'_{i-1}=\max(C_{i-1},D_{i-1})>\max(C_i,D_i)=E_i$ where the last inequality follows from $C_{i-1}>C_{i}$ and $D_{i-1}>D_{i}$. This shows that $G\geq_{\pa} E$, which completes the proof that any two elements have a join.

We now need to show that any two elements have a meet. This could be done explicitly as above. We will rather adapt the abstract argument of \cite[Proposition 3.3.1]{St97} used there in the case of a finite bounded poset.

Let $C,D \in \mathcal{P}_m$. Note first that $C,D$ always have a common lower bound $L$: for instance, let $k$ be the minimal value of all letters occurring in $C$ and $D$, and pick $L=k(k-1)\cdots (k-m+1)$. Consider the set $\mathcal{X}$ of all elements of $\mathcal{P}_m$ that lie below $C$ and $D$ and above $L$. 
Now $\mathcal{X}$ is finite and contains $L$, so by the first half of the proof we can define the join $M\coloneqq \bigvee \mathcal{X}$, i.e. the join of all elements in $\mathcal{X}$.

By construction $M\leq_{m}C,D$. We claim that it is in fact the meet of $C$ and $D$. Indeed, let $M'$ be any element below $C$ and $D$. Then $M'\leq_{m} M'\vee L$, and this join is in $\mathcal{X}$. It follows that $M'\leq_{m} M$ as wanted. We have thus shown that any two elements admit a join and a meet, and so $\mathcal{P}_m$ is a~lattice.
\end{proof}
We  now compute the M\"obius function explicitly:
\begin{proposition}
\label{prop:moebius_computation}
Let $C,D\in\mathcal{P}_m$ with $D\leq_{\pa} C$.
We have
\[
\mu(D,C)=\left\lbrace\begin{array}{ll}
(-1)^{|S_C(D)|} &\text{ if } D\in\mathbb{B}_C\\
0 &\text{ if } D\notin\mathbb{B}_C.
\end{array}\right.
\]
\end{proposition}
\begin{proof}
We apply the crosscut theorem~\cite[Corollary 3.9.4]{St97} to the interval $[D,C]$, which is a lattice by Proposition~\ref{prop:lattice}. This says that $\mu(D,C)=\sum_k(-1)^kN_k$ where $N_k$ is the number of $k$-subsets of coatoms of $[D,C]$ whose meet is $D$. Now it is easy to see that the sublattice generated by the coatoms of $[D,C]$ is precisely $\mathbb{B}_C\cap [D,C]$. So $\mu(D,C)=0$ when $D\notin \mathbb{B}_C$.

Now we notice that the poset induced by the elements of $\mathbb{B}_C$ is an upper ideal of a boolean lattice. This follows for instance by our identification of elements $D\in \mathbb{B}_C$ with subsets $S_C(D)$: one has $D\leq_{\pa} E$ for $D,E$ in $\mathbb{B}_C$ if and only $S_C(E)\subseteq S_C(D)$. This then suffices to conclude since our $\mu(D,C)$ coincides with the classical M\"obius function of the boolean lattice, which is $(-1)^{|S_C(D)|}$ in our notations.
\end{proof}

\begin{proof}[Proof of Theorem~\ref{thm:monomials_into_slides}]
Let $C\in \mathcal{P}_m$.
By definition we have
\begin{align}
\label{eq:beginning}
\backslide(C)=\sum_{D\leq_{\pa} C} \alpx(D).
\end{align}
We want to apply M\"obius inversion to~\eqref{eq:beginning}, but the sum on the right is infinite. We thus restrict to the polynomial case temporarily:
\begin{align}
\slide(C)=\sum_{D\leq_{\pa} C, D>0} \alpx(D).\end{align}
Here we take $D>0$ to mean that all components of $D$ are strictly positive.
By M\"obius inversion we get:
\begin{align}
\alpx(C)=\sum_{D\leq_{\pa} C, D>0}\mu(D,C)\,\slide(D).
\end{align}
Note that this indeed the same M\"obius function of the full poset $\mathcal{P}_m$ since imposing $D>0$ 
gives us an upper ideal, thereby preserving M\"obius values. 
By Proposition~\ref{prop:moebius_computation}, we obtain: 
\begin{equation}
\label{eq:monomials_into_slides}
\alpx(C)=\sum_{D\in \mathbb{B}_C, D>0}(-1)^{|S_C(D)|}\slide(D).
\end{equation}
This gives us the polynomial case. 

Now for $i>0$, apply \eqref{eq:monomials_into_slides} to $\gamma^i(C)=(C_1+i)\cdots (C_m+i)$, and then shift the resulting expansion by $\gamma^{-i}:x_k\mapsto x_{k-i}$. By passing to the limit in the result, we get \eqref{eq:monomials_into_backslides} (for $C>0$) since $\backslide(C)$ is the limit of $\gamma^{-i}\slide(\gamma^i(C))$ when $i$ tends to infinity. For general $C$, one has $\gamma^i(C)>0$ for $i$ large enough, and then one concludes using $\backslide(C)=\gamma^{-i}\backslide(\gamma^i(C))$.
\end{proof}

\section*{Acknowledgements}
We are grateful to Darij Grinberg for helpful comments on both content and~exposition.
\bibliographystyle{alpha}
\bibliography{Biblio_flagged}
\end{document}